\documentclass[a4paper,reqno]{amsart}

\usepackage{graphicx}
\usepackage{mathrsfs}
\usepackage{amsfonts}
\usepackage{amssymb}
\usepackage{amsmath}
\usepackage{amsthm}
\usepackage{amscd}
\usepackage[all,2cell]{xy}
\usepackage{color}
\usepackage[pagebackref,colorlinks]{hyperref}
\usepackage{enumerate}

\usepackage{hyperref}

\usepackage{comment}

\UseAllTwocells \SilentMatrices

\numberwithin{equation}{section}

\theoremstyle{plain}
\newtheorem{thm}{Theorem}[section]
\newtheorem{prop}[thm]{Proposition}

\newtheorem{lem}[thm]{Lemma}

\theoremstyle{definition}
\newtheorem{defi}[thm]{Definition}
\newtheorem{exm}[thm]{Example}
\theoremstyle{remark}

\newcommand{\D}{\mathbf{D}}
\newcommand{\K}{\mathbf{K}}
\newcommand{\ac}{\mathrm{ac}}

\def\Z{\mathbb{Z}}

\def\p{\partial}

\def\Im{{\rm Im}}

\def\La{\Lambda}
\def\I{\mathcal{I}}
\def\Bb{\mathbf{B}}

\def\s{\sharp}

\def\c{\circ}
\def\d{\delta}

\def\AA{\mathcal{A}}
\def\aa{\alpha}
\def\b{\beta}
\def\z{\zeta}

\def\g{\gamma}
\def\bu{\bullet}

\def\xra{\xrightarrow[]{}}

\renewcommand{\-}{\mbox{-}}

\newcommand{\Mod}{\mathrm{Mod}}

\newcommand{\Inj}{\mathrm{Inj}}
\newcommand{\rad}{\mathrm{rad}}

\newcommand{\Proj}{\mathrm{Proj}}

\begin{document}

\title{The injective and projective Leavitt complexes}

\author{Huanhuan Li}



\date{\today}

\begin{abstract} For a certain finite graph $E$, we consider the
corresponding finite dimensional algebra $A$ with radical square zero. An explicit compact generator for the homotopy category of acyclic complexes of injective (resp. projective) modules over $A$, called the injective (resp. projective) Leavitt complex of $E$, was constructed in \cite{li} (resp. \cite{li2}). We overview the connection between the injective (resp. projective) Leavitt complex and the Leavitt path algebra of $E$. A differential graded bimodule structure, which is right quasi-balanced, is endowed to the injective (resp. projective) Leavitt complex in \cite{li} (resp. \cite{li2}). 
We prove that the injective (resp. projective) Leavitt complex is not left quasi-balanced. 
\end{abstract}

\maketitle

\section{Introduction}

Let $A$ be a finite dimensional algebra over a field $k$. We denote by 
${\mathbf{K}_{\rm ac}}(A\-\Inj)$ the homotopy category of acyclic complexes of injective $A$-modules which is called the stable derived category of $A$ in \cite{kr}. This category is
a compactly generated triangulated category such that its subcategory of compact objects is triangle equivalent to the singularity category \cite{bu,o} of $A$.

In general, it seems very difficult to give an explicit compact generator for the stable derived category of an algebra. An explicit compact generator, called the \emph{injective Leavitt complex}, for the homotopy category $\K_{\rm ac}(A\-\Inj)$ in the case that the algebra $A$ is with radical square zero was constructed in \cite{li}. This terminology is justified by the following result: the differential graded endomorphism
algebra of the injective Leavitt complex is quasi-isomorphic to the Leavitt path algebra in the sense of \cite{ap, amp}. Here, the Leavitt path algebra is naturally $\Z$-graded and viewed as a differential graded algebra with trivial differential.

We denote by ${\K_{\ac}}(A\-\Proj)$ the homotopy category of acyclic complexes of projective $A$-modules. This category is a compactly generated triangulated category whose subcategory of compact objects is triangle equivalent to the opposite category of the singularity category  of the opposite algebra $A^{\rm op}$. An explicit compact generator, called the \emph{projective Leavitt complex}, for the homotopy category ${\K_{\ac}}(A\-\Proj)$ in the case that $A$ is an algebra with radical square zero was constructed in \cite{li2}. It is show that the opposite differential graded endomorphism algebra of the projective Leavitt complex of a finite quiver without sources is quasi-isomorphic to the Leavitt path algebra of the opposite graph \cite{li2}. 

We recall the constructions of the injective and projective Leavitt complexes. We overview the connection between the injective and projective Leavitt complex and the Leavitt path algebra of the given graph. A differential graded bimodule structure, which is right quasi-balanced, is endowed to the injective and projective Leavitt complex in \cite{li} and  \cite{li2}. We prove that the injective and projective Leavitt complex is not left quasi-balanced.

\section{Preliminaries}

In this section we recall some notations on compactly generated triangulated categories and differential graded algebras. 

\subsection{Triangulated categories} 
The construction of derived categories for an arbitrary abelian category goes back to the inspiration of Grothendieck in the early sixties in order to formulate Grothendieck's duality theory for schemes \cite{grothendieck}. The formulation in terms of triangulated categories was achieved by Verdier in the sixties and an account of this is published in Verdier  \cite{verdier}. 
 Happel \cite{happel1,happel2} investigated the derived category of bounded complexes of the category of modules over a finite dimensional algebra. Rickard  introduced the concept of tilting complex \cite{rickard1,rickard2,rickard3} in order to investigate the derived category. Now triangulated categories and derived categories have become an important tool in many branches of algebraic geometry, in algebraic analysis, non-commutative algebraic geometry, representation theory, and so on. 
 
For the definition of a triangulated category, refer to \cite[\S1.1]{happel2}, \cite[\S1.3]{neeman},etc. Homotopy categories and derived categories over an abelian category are classical examples of triangulated categories (refer to \cite[\S1.3]{happel2} etc).

We recall the triangulated structure for the homotopy category.  We first recall that a complex $X^{\bullet}=(X^{i}, d^i_{X})_{i\in \Z}$ over an abelian category $\AA$ is by definition a collection of objects $X^i$ in $\AA$ and morphisms $d^i_X:X^i\xra X^{i+1}$ in $\AA$ such that $d^{i+1}_X\circ d^i_X=0$ for all $i\in\Z$. Usually we  write a complex $X^{\bullet}=(X^{i}, d^i_{X})_{i\in \Z}$ as 
\begin{equation*}
\cdots \stackrel{}{\longrightarrow} X^{i-1}\stackrel{d^{i-1}_{X}}{\longrightarrow}X^{i}\stackrel{d_{X}^{i}}
{\longrightarrow}X^{i+1} \stackrel{}{\longrightarrow}\cdots .
\end{equation*}

Let $A$ be a finite dimensional algebra over a field $k$.  Denote by $A$-Mod the category of left $A$-modules and $\K(A\-\Mod)$ its homotopy category. For a complex $X^{\bullet}=(X^{i}, d^i_{X})_{i\in \Z}$ of $A$-modules, the complex $X^{\bullet}[1]$ is given by $(X^{\bullet}[1])^{i}=X^{i+1}$ and
$d^i_{X[1]}=-d^{i+1}_{X}$ for $i\in\Z$.
For a chain map $f^{\bullet}: X^{\bullet}\xrightarrow[]{} Y^{\bullet}$ , its \emph{mapping cone} ${\rm Con}(f^{\bullet})$ is
a complex such that ${\rm Con}(f^{\bullet})=X^{\bullet}[1]\oplus Y^{\bullet}$ with the differential
$d^i_{{\rm Con}(f^{\bullet})}=\begin{pmatrix}-d^{i+1}_{X}&0\\f^{i+1}&d^i_{Y}\end{pmatrix}.$
Each triangle in $\K(A\-\Mod)$ is isomorphic to $$\CD
  X^{\bullet} @>f^{\bullet}>> Y^{\bullet}@>{\begin{pmatrix}0\\1\end{pmatrix}} >> {\rm Con}(f^{\bullet}) @>{\begin{pmatrix}1&0\end{pmatrix}}>> X^{\bullet}[1]
\endCD$$ for some chain map $f^{\bullet}$.

\subsection{Compactly generated triangulated categories}
For a triangulated category $\mathcal{T}$, a \emph{thick} subcategory of $\mathcal{T}$
is a triangulated subcategory of $\mathcal{T}$ which is closed under direct summands. Let $\mathcal{S}$ be a class of objects in $\mathcal{T}$.
We denote by thick$\langle\mathcal{S}\rangle$ the smallest thick subcategory of $\mathcal{T}$
containing $\mathcal{S}$.
If $\mathcal{T}$ has arbitrary coproducts, we denote by
${\rm Loc}\langle\mathcal{S}\rangle$ the smallest triangulated subcategory of
$\mathcal{T}$ which contains $\mathcal{S}$ and is closed under arbitrary coproducts.
By \cite[Proposition 3.2]{bn} we have that thick$\langle\mathcal{S}\rangle\subseteq{\rm Loc}\langle \mathcal{S}\rangle$.

For a triangulated category $\mathcal{T}$ with arbitrary coproducts, an object $M$ in $\mathcal{T}$ is \emph{compact}
if the functor ${\rm Hom}_{\mathcal{T}}(M,-)$ commutes with arbitrary coproducts.
Denote by $\mathcal{T}^{c}$ the full subcategory consisting of compact objects; it is a thick subcategory.

A triangulated category $\mathcal{T}$ with arbitrary coproducts is \emph{compactly generated} \cite{ke, n1} if there exists a set $\mathcal{S}$ of compact objects such that
any nonzero object $T$ satisfies that ${\rm Hom}_{\mathcal{T}}(S,T[n])\neq 0$
for some $S\in\mathcal{S}$ and $n\in\Z$. Here, $[n]$ is the $n$-power of the shift function of $\mathcal{T}$. This is equivalent to the condition that
$\mathcal{T}={\rm Loc}\langle\mathcal{S}\rangle$, in which case we have $\mathcal{T}^{c}$=thick$\langle\mathcal{S}\rangle$; see \cite[Lemma 3.2]{n1}. If the above set $\mathcal{S}$
consists of a single object $S$, we call $S$ a \emph{compact generator} of $\mathcal{T}$.

Let $A\-\Inj$ be the category of injective $A$-modules. Denote by $\K(A\-\Inj)$ the homotopy category
of complexes of injective $A$-modules, which is a triangulated subcategory of $\K(A\-\Mod)$ that is closed under coproducts. By \cite[Proposition 2.3(1)]{kr} $\K(A\-\Inj)$ is a compactly generated triangulated category.

Denote by ${\mathbf{K}_{\rm ac}}(A\-\Inj)$ the full subcategory of $\mathbf{K}(A\-\Inj)$
formed by acyclic complexes of injective $A$-modules. 
 The homotopy category
${\mathbf{K}_{\rm ac}}(A\-\Inj)$ is called the stable derived category of $A$ in \cite{kr}. This category is
a compactly generated triangulated category such that its subcategory of compact objects is triangle equivalent to the singularity category \cite{bu,o}
of $A$. 

We denote by ${\K_{\ac}}(A\-\Proj)$ the homotopy category of acyclic complexes of projective $A$-modules. This category is a compactly generated triangulated category whose subcategory of compact objects is triangle equivalent to the opposite category of the singularity category of the opposite algebra $A^{\rm op}$.

In the last decade, Leavitt path algebras of directed graphs~\cite{ap,amp} were introduced as an
algebraisation of graph $C^*$-algebras~\cite{kprr, raeburn} and in particular
Cuntz--Krieger algebras \cite{cuntzkrieger}.

For a finite directed graph $E$, S. Paul Smith \cite{smith} describes the quotient category $${\rm QGr}(kE):={\rm Gr}(kE)/{\rm Fdim}(kE)$$ of graded $kE$-modules modulo those that are the sum of their finite dimensional submodules in terms of the category of graded modules over the Leavitt path algebra of $E^o$ over a field $k$. Here, $kE$ is the path algebra of $E$ and $E^o$ is the graph without sources or sinks that is obtained by repeatedly removing all sinks and sources from $E$. The full subcategory ${\rm qgr(kE)}$ of finitely presented objects in ${\rm QGr(kE)}$ is triangulated equivalent to the singularity category of the radical square zero algebra $kE/kE_{\geq 2}$; see \cite[Theorem 7.2]{smith}.

The homotopy categories ${\mathbf{K}_{\rm ac}}(A\-\Inj)$ and ${\K_{\ac}}(A\-\Proj)$ were described as derived categories of Leavitt path algebras, in the case that $A$ is an algebra with radical square zero associated to a certain finite directed graph; see \cite[Theorem 6.1]{cy} and \cite[Theorem 6.2]{cy}  . 

In general, it seems very difficult to give an explicit compact generator for the stable derived category of an algebra or the homotopy category of acyclic complexes of projective modules over an algebra. An explicit compact generator for
the homotopy categories $\K_{\rm ac}(A\-\Inj)$ and ${\K_{\ac}}(A\-\Proj)$, were constructed in \cite{li} and \cite{li2} respectively, in the case that the algebra $A=kE/kE_{\geq 2}$ is with radical square zero, where $E$ is a finite directed graph without sources or sinks. 

\subsection{Differential graded algebras} Differential graded (dg for short) algebras appeared in \cite{kelly}.  They found applications in the representation theory of finite-dimensional algebras in the seventies; see \cite{drozd,roiter}. The idea to use dg categories to `enhance' triangulated categories goes back at least to Bondal-Kapranov \cite{bk}, who were motivated by the study of exceptional collections of coherent sheaves on projective varieties.

We recall from \cite{ke} some notation on differential graded modules.
Let $A=\bigoplus_{n\in\Z}A^{n}$ be a $\Z$-graded algebra. For a (left) graded $A$-module $M=\bigoplus_{n\in\Z}M^n$,
elements $m$ in $M^{n}$ are said to be homogeneous of degree $n$, denoted by $|m|=n$.

A \emph{differential graded algebra} (dg algebra for short) is a $\Z$-graded algebra $A$ with a differential
$d:A \xra A$ of degree one such that $d(ab)=d(a)b+(-1)^{|a|}ad(b)$ for homogenous elements
$a,b\in A$.

A \emph{(left) differential graded} $A$-module (dg $A$-module for short) $M$
is a graded $A$-module $M=\bigoplus_{n\in \mathbb{Z}}M^{n}$ with a differential $d_{M}:M\xra M$
of degree one such that $d_{M}(a{\cdot} m)=d(a){\cdot} m+(-1)^{|a|}a{\cdot} d_{M}(m)$ for homogenous
elements $a\in A$ and $m\in M$. A morphism of dg $A$-modules is a morphism of $A$-modules
preserving degrees and commuting with differentials.
A \emph{right differential graded} $A$-module (right dg $A$-module for short) $N$
is a right graded $A$-module $N=\bigoplus_{n\in \Z}N^{n}$ with a differential $d_{N}:N\xra N$
of degree one such that $d_{N}(m\cdot a)=d_{N}(m)\cdot a+(-1)^{|m|}m\cdot d(a)$ for homogenous
elements $a\in A$ and $m\in N$. Here, we use central dots to denote the $A$-module action.

Let $B$ be another dg algebra.
Recall that a \emph{dg $A$-$B$-bimodule} $M$ is a left dg $A$-module as well as a right dg
$B$-module such that $(a\cdot m)\cdot b=a\cdot (m\cdot b)$ for $a\in A$, $m\in M$ and $b\in B$.

Let $M, N$ be (left) dg $A$-modules.
We have a $\Z$-graded vector space $${\rm Hom}_{A}(M, N)=\bigoplus_{n\in\Z}{\rm Hom}_{A}(M, N)^{n}$$
such that each component ${\rm Hom}_{A}(M, N)^{n}$ consists of $k$-linear maps
$f:M\xra N$ satisfying $f(M^i)\subseteq N^{i+n}$ for all $i\in\Z$ and $f(a\cdot m)=(-1)^{n|a|}a
\cdot f(m)$ for all homogenous elements $a\in A$.
The differential on ${\rm Hom}_{A}(M, N)$ sends $f\in
{\rm Hom}_{A}(M, N)^{n}$ to $d_{N}\circ f-(-1)^n f\circ d_{M}\in
{\rm Hom}_{A}(M, N)^{n+1}$.
Furthermore, ${\rm End}_{A}(M):={\rm Hom}_{A}(M, M)$ becomes a dg algebra
with this differential and the usual composition as multiplication.
The dg algebra ${\rm End}_{A}(M)$ is usually called the \emph{dg endomorphism algebra} of $M$.

We denote by $A^{\rm opp}$ the \emph{opposite dg algebra} of a dg algebra $A$. More precisely,
$A^{\rm opp}=A$ as graded spaces with the same differential, and the multiplication $``\circ"$
on $A^{\rm opp}$ is given by $a\circ b=(-1)^{|a||b|}ba$.

Let $B$ be another dg algebra.
Recall that a right dg $B$-module is a left dg $B^{\rm opp}$-module via $bm= (-1)^{|b||m|}m\cdot b$ for homogenous elements $b\in B, m\in M$. For a dg $A$-$B$-bimodule $M$, the canonical map
$A\xra {\rm End}_{B^{\rm opp}}(M)$ is a homomorphism of dg algebras, sending $a$ to $l_{a}$ with $ l_{a}(m)=a\cdot m$ for $a\in A$ and $m\in M$.
Similarly, the canonical map
$B\xra {\rm End}_{A}(M)^{\rm opp}$ is a homomorphism of dg algebras,
sending $b$ to $ r_{b}$ with $ r_{b}(m)=(-1)^{|b||m|}m\cdot b$ for homogenous elements $b\in B$ and $m\in M$.

A dg $A$-$B$-bimodule $M$ is called \emph{right quasi-balanced} provided that the canonical
homomorphism $B \xra {\rm End}_{A}(M)^{\rm opp}$ of dg algebras
is a quasi-isomorphism; see \cite[2.2]{cy}. Dually a dg $A$-$B$-bimodule $M$ is called \emph{left quasi-balanced} provided that the canonical homomorphism $A \xra {\rm End}_{B^{\rm opp}}(M)$ of dg algebras is a quasi-isomorphism.

Denote by $\K(A)$ the homotopy category and by $\D(A)$ the derived category
of left dg $A$-modules; they are triangulated categories with arbitrary coproducts.
For a dg $A$-$B$-bimodule
$M$ and a left dg $A$-module $N$, ${\rm Hom}_{A}(M, N)$ has a natural structure of left dg $B$-module.

Recall that ${\rm Loc}\langle M\rangle\subseteq\K(A)$ is the smallest triangulated
subcategory of
$\K(A)$ which contains $M$ and is closed under arbitrary coproducts. If
$M$ is a compact object in ${\rm Loc}\langle M\rangle$ and a dg $A$-$B$-bimodule which is right quasi-balanced, then we have a triangle equivalence \begin{equation*}{\rm Hom}_{A}(M,-):{\rm Loc}\langle M\rangle \stackrel{\sim}\longrightarrow \D(B),\end{equation*} see \cite[Proposition 2.2]{cy}; compare \cite[4.3]{ke} and \cite[Appendix A]{kr}.

\section{The injective Leavitt complex of a finite graph without sinks}

In this section, we recall the construction of the injective Leavitt complex of a finite graph without sinks and prove that the injective Leavitt complex is not left quasi-balanced.

\subsection{The injective Leavitt complex}
Let $E=(E^{0}, E^{1}; s, t)$ be a finite (directed) graph. A path in the graph $E$
is a sequence $p=\alpha_{n}\cdots\alpha_{2}\alpha_{1}$ of edges with
$t(\alpha_{j})=s(\alpha_{j+1})$ for $1\leq j\leq n-1$.
The length of $p$, denoted by $l(p)$, is $n$. The starting vertex of $p$, denoted by $s(p)$,
is $s(\alpha_{1})$.
The terminating vertex of $p$, denoted by $t(p)$, is $t(\alpha_{n})$.
We identify an
edge with a path of length one. We associate to each
vertex $i\in E^{0}$ a trivial path $e_{i}$ of length zero. Set $s(e_i)=i=t(e_i)$.
Denote by $E^{n}$ the set of all paths in $E$ of length $n$ for each $n\geq 0$. Recall that a vertex of $E$ is a sink if there is no
edge starting at it.

Let $E$ be a finite graph without sinks. For any vertex $i\in E^0$, fix an edge $\gamma$ with $s(\gamma)=i$. We call the fixed edge the \emph{special edge} starting at $i$.
For a special edge $\alpha$,
we set \begin{equation} \label{ii}
S(\alpha)=\{\beta\in E^{1}\;| \; s(\beta)=s(\alpha), \beta\neq\alpha\}.
\end{equation}
We mention that the terminology ``special edge" is taken from \cite{aajz}.

The following notion is inspired by a basis for Leavitt path algebra (refer to \cite{aajz} and \cite{hp} for the basis).

\begin{defi}\label{da}
For two paths $p=\alpha_{m}\cdots\alpha_{1}$ and $q=\beta_{n}\cdots\beta_{1}$ in $E$
with $m,n\geq 1$,
we call the pair $(p, q)$ an \emph{admissible pair} in $E$ if $t(p)=t(q)$, and either
$\alpha_{m}\neq \beta_{n}$, or $\alpha_{m}=\beta_{n}$ is not special.
For each path $r$ in $E$, we define two additional \emph{admissible pairs} $(r, e_{t(r)})$ and $(e_{t(r)}, r)$ in $E$. \hfill $\square$
\end{defi}

For each vertex $i\in E^0$ and $l\in \Z$, set
\begin{equation}
\mathbf{B}^{l}_{i}
=\{(p, q)\;|\;(p, q) \text{~is an admissible pair~} \text{with~} l(q)-l(p)=l\text{~and~}s(q)=i\}.\label{p}
\end{equation} The above set $\mathbf{B}^{l}_{i}$
is not empty for each vertex $i$ and each integer $l$; see \cite[Lemma 1.2]{li}.

Let $E$ be a finite quiver without sinks. Set
$A=kE/kE^{\geq 2}$ to be the corresponding finite dimensional algebra
with radical square zero. Indeed, $A=kE^{0}\oplus kE^{1}$
as a $k$-vector space and its Jacobson radical $\rad A=kE^{1}$ satisfying $(\rad A)^{2}=0$.

Denote by $I_{i}=D(e_{i}A)$ the injective left $A$-module for each $i\in E^{0}$, where $(e_{i}A)_{A}$ is the
indecomposable projective right $A$-module and $D={\rm Hom}_{k}(-, k)$ denotes
the standard $k$-duality.
Denote by $\{e_{i}^{\s}\}\cup\{ \alpha^{\s}| \alpha\in E^{1}, t(\alpha)=i\}$ the basis of
$I_i$, which is dual to the basis $\{e_{i}\}\cup\{
\alpha|\alpha\in E^{1}, t(\alpha)=i\}$ of $e_iA$. 

For a set $X$ and an $A$-module $M$, the coproduct $M^{(X)}$ will be
understood as
$\bigoplus_{x\in X}M\zeta_{x}$, where each component $M\zeta_{x}$ is $M$.
For an element $m\in M$, we use $m\zeta_{x}$
to denote the corresponding element in $M\zeta_{x}$.

For a path $p=\alpha_{n}\cdots\aa_2\alpha_{1}$ in $E$ of length
$n\geq 2$, we denote by $\widehat{p}=\alpha_{n-1}\cdots\alpha_{1}$ and $\widetilde{p}=\alpha_{n}\cdots\alpha_{2}$ the two \emph{truncations} of
$p$. For an edge $\aa$, denote by $\widehat{\aa}=e_{s(\aa)}$ and $\widetilde{\aa}=e_{t(\aa)}$.

\begin{defi} \label{definj} Let $E$ be a finite graph without sinks. The \emph{injective Leavitt complex} $\I^{\bullet}=(\I^{l}, \partial^{l})_{l\in \Z}$ of $E$ is defined as follows:
\begin{enumerate}\item[(1)] the $l$-th component $\I^{l}=\bigoplus_{i\in E^{0}}{I_{i}}^{(\Bb^l_i)}$;

\item[(2)] the differential $\partial^{l}:\I^{l}\xrightarrow[]{} \I^{l+1}$ is given by
$\partial^{l}(e_{i}^{\s}\zeta_{(p, q)})=0$ and
\begin{equation*}\partial^{l}(\alpha^{\s}\zeta_{(p, q)})=\begin{cases}
 e_{s(\alpha)}^{\s}\zeta_{(\widehat{p}, e_{s(\alpha)})}-\sum\limits_{\beta\in S(\alpha)} e_{s(\alpha)}^{\s}\zeta_{(\beta\widehat{p}, \beta)}, &
 \begin{matrix}\text{if~} q=e_{i}, ~~p=\alpha \widehat{p} \\\text{~and~} \alpha \text{~is special};\end{matrix}\\
 e_{s(\alpha)}^{\s}\zeta_{(p, q\alpha)},& \text{otherwise},
\end{cases}
\end{equation*}
for any $i\in E^{0}$, $(p, q)\in \Bb^l_i$ and $\alpha\in E^{1}$ with $t(\alpha)=i$. Here, the set $S(\alpha)$ is defined in $(\ref{ii})$. \hfill $\square$
\end{enumerate}
\end{defi} Each component $\I^l$ is an injective $A$-module. The differentials $\p^l$ are $A$-module morphisms. Indeed, $\I^{\bu}$ is an acyclic complex of injective $A$-modules; see \cite[Proposition 1.9]{li}.

\begin{exm} \label{exm} Let $E$ be the following graph with one vertex and $n$ loops with $n\geq 2$.  $$
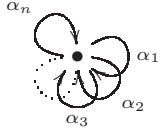
\xymatrix{
\!\!\!&  \bullet \ar@{.}@(l,d) \ar@(ur,dr)^{\alpha_{1}} \ar@(r,d)^{\alpha_{2}} \ar@(dr,dl)^{\alpha_{3}} 
\ar@(l,u)^{\alpha_{n}}& 
}$$ We choose $\alpha_{1}$ to be the special edge.
Let $e$ be the trivial path corresponding the unique vertex. Set $\mathbf{B}^{l}$ to be the set of admissible pairs $(p, q)$ in $E$ with $l(q)-l(p)=l$ for each $l\in \mathbb{Z}$.
A pair $(p,q)$ of paths lies in $\Bb^l$ if and only if
$l(q)-l(p)=l$ and $p, q$ do not end with $\aa_1$ simultaneously.
In particular, the set $\Bb^l$ is infinite.

The corresponding algebra $A$ with radical square zero
has a $k$-basis $\{e, \aa_1,\cdots, \aa_n\}$.
Set $I=D(A_{A})$.
Then the injective Leavitt complex $\I^{\bullet}$ of $E$ is as follows.
$$
\cdots
\stackrel{}{\longrightarrow} I^{({\mathbf{B}^{-1}})}
\stackrel{\p^{-1}}{\longrightarrow} I^{({\mathbf{B}^{0}})}
\stackrel{\p^{0}}{\longrightarrow} I^{({\mathbf{B}^{1}})}
\stackrel{}{\longrightarrow} \cdots
$$ We write the differential $\p^{-1}$ explicitly: $\p^{-1}(e^{\s}\zeta_{(p, q)})=0$, $\p^{-1}(\alpha_{i}^{\s} \zeta_{(p, q)})=e^{\s}\zeta_{(p, q\alpha_{i})}$ for $2\leq i\leq n$ and

\begin{equation*}\p^{-1}(\alpha_{1}^{\s}\zeta_{(p, q)})=\begin{cases}
 e^{\s}\zeta_{(e, e)}-\sum_{i=2}^ne^{\s}\zeta_{(\alpha_{i}, \alpha_{i})},  &\text{if $q=e$ and $p=\alpha_{1}$};\\
 e^{\s}\zeta_{(p, q\alpha_{1})},& \text{otherwise},
\end{cases}
\end{equation*} for $(p, q)\in \mathbf{B}^{-1}$.
\end{exm}

 Recall that the Leavitt path algebra $L_{k}(E)$ of a finite graph $E$ is the $k$-algebra generated by the set $\{v\;|\;v\in E^{0}\}\cup \{e\;|\;e\in E^{1}\}
\cup\{e^{*}\;|\;e\in E^{1}\}$ subject to the following relations:
\begin{itemize}
\item [(0)] $vw=\delta_{v, w}v$ for every $v, w\in E^{0}$;

\item[(1)] $t(e)e=es(e)=e$ for all $e\in E^{1}$;

\item[(2)] $e^{*}t(e)=s(e)e^{*}=e^{*}$ for all $e\in E^{1}$;

\item[(3)] $ef^{*}=\delta_{e, f}t(e)$ for all $e,f\in E^{1}$;

\item[(4)] $\sum_{\{e\in E^{1}\;|\;s(e)=v\}}e^{*}e=v$ for every $v\in E^{0}$ which is not a sink.
 \end{itemize} 
 
 Here, $\delta$ is the Kronecker symbol. The relations $(3)$ and $(4)$ are called
\emph{Cuntz-Krieger relations}. The relation $(3)$ is called (CK1)-relation and $(4)$ is called (CK2)-relation. The elements $\alpha^{*}$ for $\alpha\in E^{1}$ are called \emph{ghost arrows}.

 If $p=e_{n}\cdots e_2e_{1}$ is a path in $E$ of length $n\geq 1$, we define $p^{*}=e_{1}^*e_2^*\cdots e_{n}^*$.
For convention, we set $v^*=v$ for $v\in E^{0}$.
We observe by $(2)$ that for paths $p, q$ in $E$, $p^*q=0$ for $t(p)\neq t(q)$. 
 
 Recall that the Leavitt path algebra is naturally $\Z$-graded by the length of paths. In what follows, we write $B=L_k(E)^{\rm op}$, which is the opposite algebra of $L_k(E)$. Then $B$ is a $\Z$-graded algebra. We view $B$ as a dg algebra with trivial differential.

Consider $A=kE/kE^{\geq 2}$ as a dg algebra concentrated on degree zero. Recall the injective Leavitt complex $\I^{\bu}=\bigoplus_{l\in\Z}\I^l$, which is a left dg $A$-module.
By \cite[Proposition 3.6]{li}, $\I^{\bu}$ is a right dg $B$-module and a dg $A$-$B$-bimodule. 

\vskip 5pt

The following theorem demonstrates the role of the injective Leavitt complex in the stable category and establishes a connection between the injective Leavitt complex and the Leavitt path algebra, which justifies the terminology.

\noindent $\mathbf{Theorem~I}$ \emph{Let $E$ be a finite graph without sinks, and let $A=kE/kE^{\geq 2}$ be the corresponding algebra with radical square zero.}

\begin{enumerate}
\item[(1)] \emph{The injective Leavitt complex $\mathcal{I}^{\bullet}$ of $E$ is a
compact generator for the homotopy category $\K_{\rm ac}(A\-\Inj)$.}

\item[(2)] \emph{The dg $A$-$B$-bimodule
$\mathcal{I}^{\bullet}$ is right quasi-balanced. In particular, the dg endomorphism algebra ${\rm End}_{A}(\I^{\bu})$ is quasi-isomorphic to the Leavitt path algebra $L_k(E)$. Here, $L_k(E)$ is naturally $\Z$-graded and viewed as a dg algebra with trivial differential.}
\end{enumerate}  For the proof of Theorem I, refer to \cite[Theorem 2.13]{li} and \cite[Theorem 4.2]{li}.

\vskip 5pt

There is a unique right $B$-module morphism $\psi: B\xra \I^{\bullet}$ with $$\psi(1)=\sum_{i\in E^0}e_i^{\s}\z_{(e_i, e_i)}.$$ Here, $1$ is the unit of $B$.
For each edge $\beta\in E^{1}$, there is a unique right $B$-module morphism $\psi_{\b}: B\xra \I^{\bullet}$ with $\psi_{\b}(1)=\b^{\s}\z_{(e_{t(\b)}, e_{t(\b)})}$. Let $(p, q)$ be an admissible pair in $E$. We have $\sum_{i\in E^0}e_i^{\s}\z_{(e_i, e_i)}\cdot p^*q=e_{s(q)}^{\s}\z_{(p, q)}$ and $\b^{\s}\z_{(e_i, e_i)}\cdot p^*q=\delta_{i, s(q)}\b^{\s}\z_{(p, q)}$ for each edge $\b\in E^1$ with $t(\b)=i$. It follows that \begin{equation}\label{psipsi}\psi(p^*q)=e_{s(q)}^{\s}\z_{(p, q)}
\text{\qquad and\qquad} \psi_{\b}(p^*q)=\delta_{s(q), t(\b)}\b^{\s}\z_{(p, q)}.\end{equation} It follows that $\psi$ is injective and $\Im \psi_{\b}\cong e_{t(\b)}B$ for each $\b\in E^1$. Consider the gradings of $B$ and $\I^{
\bu}$. We have that $\psi$ and $\psi_{\b}$ for $\b\in E^1$ are right graded $B$-module morphisms.

Set $\Psi=\begin{pmatrix}
\psi&(\psi_{\b})_{{\b\in E^1}}
\end{pmatrix}: B\oplus B^{(E^1)}\xra \I^{\bullet}$. The map $\Psi$ is a graded right $B$-module morphism.

\begin{lem} \label{property}
The above morphism $\Psi$ is surjective and the injective Leavitt complex $\I^{\bullet}$ of $E$ is a graded projective right $B$-module.
\end{lem}
\begin{proof}
For each $i\in E^0$, $l\in \Z$ and $(p, q)\in \Bb^l_i$, we have that
$e_i^{\s}\z_{(p, q)}=\psi(p^*q)$ and $\aa^{\s}\z_{(p, q)}=\psi_{\aa}(p^*q)$
for each edge $\aa\in E^1$ with $t(\aa)=i$. Then $\Psi$ is surjective and $\I^{\bullet}=\Im \Psi$.
Observe that $\Im \Psi=\Im\psi+\sum_{\b\in E^1}\Im \psi_{\b}=
\Im\psi\oplus(\oplus_{\b\in E^1}\Im \psi_{\b})$. Then $\Im \Psi\cong B\oplus(\oplus_{\b\in E^1}e_{t(\b)}B)$ and we are done.
\end{proof}

The following observation implies that the dg $A$-$B$-bimodule $\I^{\bu}$ is not left quasi-balanced. In other words, the canonical dg algebra homomorphism $A \xra {\rm End}_{B^{\rm opp}}(\mathcal{I}^{\bullet})$ is not a quasi-isomorphism. Indeed, we infer from the observation that the dg endomorphism algebra ${\rm End}_{B^{\rm opp}}(\mathcal{I}^{\bullet})$ is acyclic.

\begin{prop}\label{leftquasi} We have that $\I^{\bu}=0$ in the homotopy category of right dg $B$-modules.
\end{prop}

\begin{proof} Recall from $(\ref{psipsi})$ the right $B$-module morphisms
$\psi$ and $\psi_{\b}$ for $\b\in E^1$. For each $l\in \Z$, the set
$\{e_i^{\s}\z_{(p, q)}, \aa^{\s}\z_{(p, q)}\;|\;i\in E^0, (p, q)\in \Bb^l_i \;\text{and}\; \aa\in E^1 \;\text{with} \;t(\aa)=i\}$ is a $k$-basis of $\I^{l}$.
We define a $k$-linear map $h^l:\I^{l}\xra \I^{l-1}$ such that
\begin{equation*}
\begin{cases}
h^l(\aa^{\s}\z_{(p, q)})=0;& \\
h^l(e_i^{\s}\z_{(p, q)})=\sum_{\{\b\in E^1\;|\;s(\b)=i\}}\psi_{\b}(p^*q\b^*),&
\end{cases}
\end{equation*} for each $i\in E^{0}$, $l\in \Z$,
$(p, q)\in \mathbf{B}^{l}_{i}$, and $\aa\in E^1$ with $t(\aa)=i$. Here,
$p^*q\b^*$ is the multiplication of $p^*q$ and $\b^*$ in $L_k(E)$.
For any element $b\in L_k(E)^le_i$, we have $$h^l(\psi(b))=\sum_{\{\b\in E^1\;|\; s(\b)=i\}}\psi_{\b}(b\b^*).$$ Recall that $\psi$ and $\psi_{\b}$ for $\b\in E^1$ are right $B$-module morphisms. Then we have that $h=(h^l)_{l\in\Z}$ is right graded $B$-module morphisms of degree $-1$. In other words, we have $h\in {\rm End}_{B^{\rm opp}}(\I^{\bu})^{-1}$.

It suffices to prove that $\p^{l-1}\c h^l+h^{l+1}\c \p^{l}={\rm Id}_{\I^{l}}$ for each $l\in\Z$. For each $i\in E^{0}$, $l\in \Z$ and $(p, q)\in \mathbf{B}^{l}_{i}$, we have that
\begin{equation*}
\begin{split}
&(\p^{l-1}\c h^l+h^{l+1}\c \p^{l})(e_i^{\s}\z_{(p, q)})\\
&=\sum\nolimits_{\{\b\in E^1\;|\;s(\b)=i\}}(\p^{l-1}\c\psi_{\b})(p^*q\b^*)\\
&=\sum\nolimits_{\{\b\in E^1\;|\;s(\b)=i\}}\psi(p^*q\b^*\b)\\
&=e_i^{\s}\z_{(p, q)},
\end{split}
\end{equation*} where the second equality uses \cite[Lemma 3.7]{li}
and the last equality uses the (CK2)-relation for the Leavitt path algebra.
Similarly, we have
\begin{equation*}
\begin{split}
&(\p^{l-1}\c h^l+h^{l+1}\c \p^{l})(\aa^{\s}\z_{(p, q)})\\
&=h^{l+1}((\p^{l}\c\psi_{\aa})(p^*q))\\
&=h^{l+1}(\psi(p^*q\aa))\\
&=\sum\nolimits_{\{\b\in E^1\;|\;s(\b)=s(\aa)\}}\psi_{\b}(p^*q\aa\b^*)\\
&=\aa^{\s}\z_{(p, q)}
\end{split}
\end{equation*} for $\aa\in E^1$ with $t(\aa)=i$. Here, the last equality uses the (CK1)-relation for the Leavitt path algebra.
\end{proof}

\subsection{The independence of the injective Leavitt complex}
We show that the definition of the injective Leavitt complex of $E$ is
independent of the choice of special edges in $E$.

Denote by $S$ and $S'$ two different sets of special edges of $E$. For each $i\in E^0$ and $l\in\Z$,
let $\Bb^l_i$ and $(\Bb^l_i)'$ be the corresponding sets of admissible pairs with
respect to $S$ and $S'$ respectively; see $(\ref{p})$.
Define a map $\Phi_{li}:\Bb^l_i\xra(\Bb^l_i)'$ such that for $(p, q)\in\Bb^l_i$
\begin{equation*}\Phi_{li}((p,q))=\begin{cases}
(\aa \widehat{p}, \aa \widehat{q}) , & \text{if $p=\aa' \widehat{p}$, $q=\aa' \widehat{q}$ and $\aa'\in S'$};\\
(p, q),& \text{otherwise},
\end{cases}
\end{equation*} where $\aa\in S$ with $s(\aa)=s(\aa')$.
The map $\Phi_{li}:\Bb^l_i\xra ({\Bb^l_i})'$ is a bijection.
In fact, the inverse map of $\Phi_{li}$ can be defined symmetrically.

Denote by $\I^{\bullet}=(\I^{l}, \partial^{l})_{l\in \Z}$ and ${\I^{\bullet}}'=({\I^{l}}', (\partial^{l})')_{l\in \Z}$ the injective Leavitt complexes of $E$ with $S$ and $S'$
sets of special edges, respectively.
Let $\La'$ be the $k$-basis of $L_k(E)$ given by \cite[Theorem 1]{aajz} and \cite[Corollary 17]{hp} with $S'$ the set of special edges.

Recall that $B=L_k(E)^{\rm op}$. The maps $\psi': B\xra {\I^{\bullet}}', p^*q \mapsto e_{s(q)}^{\s}\z_{(p, q)}$ and
$\psi'_{\b}: B\xra {\I^{\bullet}}', p^*q \mapsto \delta_{s(q), t(\b)}\b^{\s}\z_{(p, q)}$
for $p^*q\in\La'$ and $\beta\in E^{1}$ are graded right $B$-module morphisms.

For each $l\in\Z$, define a $k$-linear map $\omega^l: \I^{l}\xra(\I^{l})'$ such that
$\omega^l(e_i^{\s}\zeta_{(p, q)})=\psi'(p^*q)$ and $\omega^l(\aa^{\s}\zeta_{(p, q)})=\psi'_{\aa}(p^*q)$ for $i\in E^{0}$ and
$(p, q)\in \mathbf{B}^{l}_{i}$ and $\aa\in E^1$
with $t(\aa)=i$.
Let $\omega^{\bullet}=(\omega^l)_{l\in \Z}:\I^{\bullet}\xra{\I^{\bullet}}'$.
By definitions we have $\omega^{\bullet}\circ\psi=\psi'$ and $\omega^{\bullet}\circ\psi_{\b}=\psi_{\b}'$ for each edge $\b\in E^1$.
Then we have $\omega^{\bullet}\circ\Psi=\Psi'$ with $\Psi'=\begin{pmatrix}
\psi'&(\psi'_{\b})_{{\b\in E^1}}
\end{pmatrix}:B\oplus B^{(E^1)}\xra {\I^{\bullet}}'$ a morphism of graded right $B$-modules.

\begin{prop} \label{unique} $(1)$ The above map $\omega^{\bullet}:\I^{\bullet} \xra{\I^{\bullet}}'$ is an isomorphism of complexes of $A$-modules.

$(2)$ The above map $\omega^{\bullet}:\I^{\bullet} \xra {\I^{\bullet}}'$ is an isomorphism of right dg $B$-modules.
\end{prop}

\begin{proof} $(1)$  We can directly check that $\omega^l$ is an $A$-module map for each $l\in\Z$.
The inverse map of $\omega^l$ can be defined symmetrically. We have that $\omega^l$ is an isomorphism of $A$-modules for each $l\in\Z$.
It remains to prove that $\omega^{\bullet}$ is a chain map of complexes.
For each $i\in E^{0}$, $l\in \Z$ and
$(p, q)\in \mathbf{B}^{l}_{i}$, since $((\p^l)'\circ\omega^i)(e_i^{\s}\z_{(p, q)})=0=(\omega^{l+1}\circ\p^l)(e_i^{\s}\z_{(p, q)})$,
it suffices to prove that $((\p^l)'\circ\omega^l)(\aa^{\s}\z_{(p, q)})=(\omega^{l+1}\circ\p^l)(\aa^{\s}\z_{(p, q)})$
for $\aa\in E^1$ with $t(\aa)=i$. By \cite[Lemma 3.7]{li}, we have that
$((\p^l)'\circ\omega^l)(\aa^{\s}\z_{(p, q)})=(\p^l)'(\psi'_{\aa}(p^*q))=\psi'(p^*q\aa)$
and $(\omega^{l+1}\circ\p^l)(\aa^{\s}\z_{(p, q)})=\omega^{l+1}(\psi(p^*q\aa))=\psi'(p^*q\aa)$.
Then we are done.

$(2)$ It remains to prove that $\omega^{\bullet}$ is a graded right $B$-module morphism.
By Lemma \ref{property}, there exists a graded right $B$-module morphism $\Xi: \I^{\bullet}\longrightarrow B\oplus B^{(E^1)}$
such that $\Psi\circ\Xi=\rm{Id}_{\I^{\bullet}}$. Since $\omega^{\bullet}\circ \Psi=\Psi'$, we have that
$\omega^{\bullet}=\omega^{\bullet}\circ (\Psi\circ\Xi)=\Psi'\circ\Xi$ is a composition of graded right $B$-module morphisms.
\end{proof}

\section{The projective Leavitt complex of a finite graph without sources}

In this section, we recall the construction of the projective Leavitt complex of a finite graph without sources and prove that the projective Leavitt complex is not left quasi-balanced.

\subsection{The projective Leavitt complex}

Let $E$ be a finite graph without sources. For a vertex $i\in E^0$, fix an edge $\gamma$
with $t(\gamma)=i$. We call the fixed edge the \emph{associated edge} terminating at $i$.
For an associated edge $\alpha$,
we set \begin{equation} \label{eq:vv}
T(\alpha)=\{\beta\in E^{1}\;| \;t(\beta)=t(\alpha), \beta\neq\alpha\}.
\end{equation}

\begin{defi}
For two paths $p=\aa_{m}\cdots\aa_{2}\aa_{1}$ and $q=\b_{n}\cdots \b_{2}\b_{1}$ with $m,n\geq 1$, we call the pair $(p, q)$ an
\emph{associated pair} in $E$
if $s(p)=s(q)$, and either $\alpha_{1}\neq\beta_{1}$, or  $\alpha_{1}=\beta_{1}$ is not associated.  In addition, we call $(p, e_{s(p)})$ and $(e_{s(p)}, p)$ \emph{associated pairs} in $E$ for each path $p$ in $E$. \hfill $\square$
\end{defi}

For each vertex $i\in Q_{0}$ and $l\in \Z$, set
\begin{equation}
\label{p}
\mathbf{\Lambda}^{l}_{i}=\{(p, q)\;|\;(p, q) \text{~is an associated pair with~}l(q)-l(p)=l \text{~and~} t(p)=i\}.
\end{equation} The set $\mathbf{\Lambda}^{l}_{i}$ is not empty for each vertex $i$ and each integer $l$; see  \cite[Lemma 2.2]{li2}.

Recall that $A=kE/kE^{\geq 2}$ is a finite dimensional algebra with radical square zero. Denote by $P_{i}=Ae_{i}$ the indecomposable projective left $A$-module for $i\in E^{0}$.

\begin{defi} \label{defproj} Let $E$ be a finite quiver without sources. The \emph{projective Leavitt complex} $\mathcal{P}^{\bullet}=(\mathcal{P}^{l}, \delta^{l})_{l\in\Z}$ of $E$ is defined as follows:

\noindent $(1)$ the $l$-th component $\mathcal{P}^{l}=\bigoplus_{i\in E^{0}}{P_{i}}^{(\mathbf{\La}^{l}_{i})}$.

\noindent $(2)$ the differential $\delta^{l}: \mathcal{P}^{l}\longrightarrow \mathcal{P}^{l+1}$ is given by  $\delta^{l}(\alpha\zeta_{(p, q)})=0$ and \begin{equation*}\delta^{l}(e_{i}\zeta_{(p, q)})=\begin{cases}
 \beta\zeta_{(\widehat{p},q)}, & \text{if $p=\beta\widehat{p}$};\\
\sum_{\{\beta\in E^{1}\;|\;  t(\beta)=i\}}\beta\zeta_{(e_{s(\beta)}, q\beta)},& \text{if $l(p)=0$,}
\end{cases}
\end{equation*}
for any $i\in E^{0}$, $(p, q)\in \mathbf{\La}^{l}_{i}$
and $\alpha\in E^{1}$ with $s(\aa)=i$. \hfill $\square$
\end{defi}

Each component $\mathcal{P}^{l}$ is a projective $A$-module and the differential $\d^{l}$ are $A$-module morphisms. $\mathcal{P}^{\bu}$ is an acyclic complex of injective $A$-modules; see \cite[Proposition 2.7]{li2}.

\begin{exm} \label{exm} Let $E$ be the following graph with one vertex and $n$ loops with $n\geq 2$.  $$
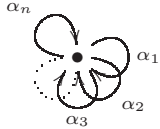
\xymatrix{
\!\!\!&  \bullet \ar@{.}@(l,d) \ar@(ur,dr)^{\alpha_{1}} \ar@(r,d)^{\alpha_{2}} \ar@(dr,dl)^{\alpha_{3}} 
\ar@(l,u)^{\alpha_{n}}& 
}$$ We choose $\alpha_{1}$ to be the associated edge.
Let $e$ be the trivial path corresponding the unique vertex. Set $\mathbf{\Lambda}^{l}$ to be the set of associated pairs $(p, q)$ in $E$ with $l(q)-l(p)=l$ for each $l\in \mathbb{Z}$.
A pair $(p,q)$ of paths lies in $\mathbf{\Lambda}^{l}$ if and only if
$l(q)-l(p)=l$ and $p, q$ do not begin with $\aa_1$ simultaneously.
In particular, the set $\mathbf{\Lambda}^{l}$ is infinite.

The corresponding algebra $A$ with radical square zero
has a $k$-basis $\{e, \aa_1,\cdots, \aa_n\}$.
Set $P=A$.
Then the projective Leavitt complex $\mathcal{P}^{\bullet}$ of $E$ is as follows.
$$
\cdots
\stackrel{}{\longrightarrow} P^{({\mathbf{\Lambda}^{0}})}
\stackrel{\d^{0}}{\longrightarrow}P^{({\mathbf{\Lambda}^{1}})}
\stackrel{\d^{1}}{\longrightarrow} P^{({\mathbf{\Lambda}^{2}})}
\stackrel{}{\longrightarrow} \cdots
$$ We write the differential $\delta^{1}$ explicitly: $\d^{1}(e\zeta_{(p, q)})=\begin{cases}
 \beta\zeta_{(\widehat{p}, q)},  &\text{if $p=\beta \widehat{p}$};\\
 \sum_{i=1}^n\aa_i\zeta_{(e, q\alpha_{i})},& \text{if $l(p)=0$},
\end{cases}
$, $\d^{1}(\alpha_{j} \zeta_{(p, q)})=0$ for $1\leq j\leq n$ and $(p, q)\in {\mathbf{\Lambda}}^{1}$.
\end{exm}

For notation, $E^{\rm op}$
is the opposite graph of $E$. For a path $p$ in $E$, denote by
$p^{\rm op}$ the corresponding path in $E^{\rm op}$. The starting and terminating vertices of $p^{\rm op}$
are $t(p)$ and $s(p)$, respectively.
For convention, $e_{j}^{\rm op}=e_{j}$ for each vertex $j\in E^{0}$.

In what follows, we write $B=L_k(E^{\rm op})$, which is a $\Z$-graded algebra. We view $B$ as a dg algebra with trivial differential.

Consider $A=kE/kE^{\geq 2}$ as a dg algebra concentrated on degree zero. The projective Leavitt complex $\mathcal{P}^{\bu}=\bigoplus_{l\in\Z}\mathcal{P}^l$ is a left dg $A$-module.
By \cite[Proposition 4.6]{li2}, $\mathcal{P}^{\bu}$ is a right dg $B$-module and a dg $A$-$B$-bimodule. 

\vskip 5pt

The following theorem establishes a connection between the projective Leavitt complex and the Leavitt path algebra, which justifies the terminology.

\noindent $\mathbf{Theorem ~~II}$ \emph{Let $E$ be a finite graph without sources, and let $A=kE/kE^{\geq 2}$ be the corresponding algebra with radical square zero.}

\begin{enumerate}
\item[(1)] \emph{The projective Leavitt complex $\mathcal{P}^{\bullet}$ of $E$ is a compact generator for the homotopy category $\K_{\rm ac}(A\-\Proj)$.}

\item[(2)] \emph{The dg $A$-$B$-bimodule
$\mathcal{P}^{\bullet}$ is right quasi-balanced. In particular, the opposite dg endomorphism algebra of the projective Leavitt complex of $E$ is quasi-isomorphic to the Leavitt path algebra $L_k(E^{\rm op})$.  Here, $E^{\rm op}$ is the opposite graph of $E$; $L_k(E^{\rm op})$ is naturally $\Z$-graded and viewed as a dg algebra with trivial differential.}
\end{enumerate} For the proof of Theorem II, refer to \cite[Theorem 3.7]{li2} and \cite[Theorem 5.2]{li2}.

\vskip 5pt

There is a unique right $B$-module morphism $\phi: B\xra \mathcal{P}^{\bullet}$ with $$\phi(1)=\sum_{i\in E^0}e_i\z_{(e_i, e_i)}.$$ Here, $1$ is the unit of $B$. For each edge $\beta\in E^{1}$, there is a unique right $B$-module morphism $\phi_{\b}: B\xra \mathcal{P}^{\bullet}$ with $\phi_{\b}(1)=\b\z_{(e_{s(\b)}, e_{s(\b)})}$. Let $(p, q)$ be an associated pair in $E$. We have $\sum_{i\in E^0}e_i\zeta_{(e_{i}, e_i)}\cdot (p^{\rm op})^*q^{\rm op}=e_{t(p)}\zeta_{(p,  q)};$ and $\b\zeta_{(e_{s(\b)}, e_{s(\b)})}\cdot (p^{\rm op})^*q^{\rm op}=\d_{s(\b), t(p)}\b\zeta_{(p, q)}$ for each edge $\b\in E^1$; see \cite[Lemma 4.5]{li2}. It follows that \begin{equation}\label{phi}\phi((p^{\rm op})^*q^{\rm op})=e_{t(p)}\z_{(p, q)}
\text{\qquad and\qquad} \phi_{\b}((p^{\rm op})^*q^{\rm op})=\delta_{t(p), s(\b)}\b\z_{(p, q)}.\end{equation} It follows that $\phi$ is injective and $\Im \phi_{\b}\cong e_{s(\b)}B$ for each $\b\in E^1$. Consider the gradings of $B$ and $\mathcal{P}^{
\bu}$. We have that $\phi$ and $\phi_{\b}$ for $\b\in E^1$ are right graded $B$-module morphisms.

Set $\Phi=\begin{pmatrix}
\phi&(\phi_{\b})_{{\b\in E^1}}
\end{pmatrix}: B\oplus B^{(E^1)}\xra \mathcal{P}^{\bullet}$. The map $\Phi$ is a graded right $B$-module morphism.

\begin{lem} \label{pproj}
The above morphism $\Phi$ is surjective and the projective Leavitt complex $\mathcal{P}^{\bullet}$ of $E$ is a graded projective right $B$-module.
\end{lem}
\begin{proof}
For each $i\in E^0$, $l\in \Z$ and $(p, q)\in \mathbf{\Lambda}^l_i$, we have that
$e_i\z_{(p, q)}=\phi((p^{\rm op})^*q^{\rm op})$ and $\aa\z_{(p, q)}=\phi_{\aa}((p^{\rm op})^*q^{\rm op})$
for each edge $\aa\in E^1$ with $s(\aa)=i$. Then $\Phi$ is surjective and $\mathcal{P}^{\bullet}=\Im \Phi$.
Observe that $\Im \Phi=\Im\phi+\sum_{\b\in E^1}\Im \phi_{\b}=
\Im\phi\oplus(\oplus_{\b\in E^1}\Im \phi_{\b})$. Then $\Im \Phi\cong B\oplus(\oplus_{\b\in E^1}e_{s(\b)}B)$ and we are done.
\end{proof}

We prove that the dg $A$-$B$-bimodule $\mathcal{P}^{\bu}$ is not left quasi-balanced. In other words, the canonical dg algebra homomorphism $A \xra {\rm End}_{B^{\rm opp}}(\mathcal{P}^{\bullet})$ is not a quasi-isomorphism. Indeed, we infer from the observation that the dg endomorphism algebra ${\rm End}_{B^{\rm opp}}(\mathcal{P}^{\bullet})$ is acyclic.

\begin{prop}\label{leftquasi} We have that $\mathcal{P}^{\bu}=0$ in the homotopy category of right dg $B$-modules.
\end{prop}

\begin{proof} Recall from $(\ref{phi})$ the right $B$-module morphisms
$\phi$ and $\phi_{\b}$ for $\b\in E^1$. For each $l\in \Z$, the set
$\{e_i\z_{(p, q)}, \aa\z_{(p, q)}\;|\;i\in E^0, (p, q)\in \mathbf{\Lambda}^l_i \;\text{and}\; \aa\in E^1 \;\text{with} \; s(\aa)=i\}$ is a $k$-basis of $\mathcal{P}^{l}$.
We define a $k$-linear map $h^l:\mathcal{P}^{l}\xra \mathcal{P}^{l-1}$ such that $h^l(e_i\z_{(p, q)})=0$ and 
\begin{equation*}
h^l(\aa\z_{(p, q)})=\begin{cases}
e_{t(\aa)}\zeta_{(e_{t(\aa)}, \widetilde{q})}-\sum_{\b\in T(\aa)}e_{t(\aa)}\zeta_{(\b, \widetilde{q}\b)};&  \begin{matrix}\text{if~} l(p)=0, ~~q=\widetilde{q}\aa \\\text{~and~} \alpha \text{~is associated};\end{matrix}\\
e_{t(\aa)}\zeta_{(\aa p, q)}, & \text{otherwise.}
\end{cases}
\end{equation*} for each $i\in E^{0}$, $l\in \Z$,
$(p, q)\in \mathbf{\Lambda}^{l}_{i}$, and $\aa\in E^1$ with $s(\aa)=i$. It follows that $h^l(\phi_{\aa}((p^{\rm op})^*q^{\rm op}))=\phi((p^{\rm op})^*q^{\rm op}).$ For any element $b\in e_iL_k(E)^l$, we have \begin{equation}\label{uu}h^l(\phi_{\aa}(b)=\phi((\b^{\rm op})^*b).\end{equation} Here,
$(\b^{\rm op})^*b$ is the multiplication of $b$ and $(\b^{\rm op})^*$ in $B=L_k(E^{\rm op})$. Recall that $\phi$ and $\phi_{\b}$ for $\b\in E^1$ are right $B$-module morphisms. Then we have that $h=(h^l)_{l\in\Z}$ is right graded $B$-module morphisms of degree $-1$. In other words, we have $h\in {\rm End}_{B^{\rm opp}}(\mathcal{P}^{\bu})^{-1}$.

It suffices to prove that $\d^{l-1}\c h^l+h^{l+1}\c \d^{l}={\rm Id}_{\mathcal{P}^{l}}$ for each $l\in\Z$. For each $i\in E^{0}$, $l\in \Z$ and $(p, q)\in \mathbf{\Lambda}^{l}_{i}$, we have that
\begin{equation*}
\begin{split}
&(\d^{l-1}\c h^l+h^{l+1}\c \d^{l})(e_i\z_{(p, q)})\\
&=0+h^{l+1}(\d^{l}(e_i\z_{(p, q)}))\\
&=\begin{cases}
h^{l+1}(\b\zeta_{(p, q)});& \text{if $p=\b\widehat{p}$;} \\
h^{l+1}(\sum_{\{\g\in E^1 \;|\; t(\g)=i\}}\g\zeta_{(e_{s(\g)}, q\g)}), & \text{if $l(p)=0$.}
\end{cases}\\
&=e_{i}\zeta_{(p, q)}\\
\end{split}
\end{equation*} where the last equality uses the definition of $h^{l+1}$. Similarly, we have
\begin{equation*}
\begin{split}
&(\d^{l-1}\c h^l+h^{l+1}\c \d^{l})(\aa\z_{(p, q)})\\
&=(\d^{l-1}\c h^l)(\aa\z_{(p, q)})\\
&=\d^{l-1}\big((h^{l}\c\phi_{\aa})((p^{\rm op})^*q^{\rm op})\big)\\
&=\d^{l-1}(\phi((\aa^{\rm op})^*(p^{\rm op})^*q^{\rm op}))\\
&=\sum_{\b\in E^1\;|\; t(\b)=t(\aa)}\phi_{\b}(\b^{\rm op}(\aa^{\rm op})^*(p^{\rm op})^*q^{\rm op})\\
&=\aa\z_{(p, q)}
\end{split}
\end{equation*} for $\aa\in E^1$ with $s(\aa)=i$. Here, the third equality uses \eqref{uu}, the fourth equality uses \cite[Lemma 4.7]{li2}, and the second last equality uses the (CK1)-relation for the Leavitt path algebra $L_k(E^{\rm op})$.
\end{proof}

\subsection{The independence of the projective Leavitt complex}
We show that the definition of the projective Leavitt complex of $E$ is
independent of the choice of associated edges in $E$.

Denote by $H$ and $H'$ two different sets of associated edges of $E$. For each $i\in E^0$ and $l\in\Z$,
let $\mathbf{\Lambda}^l_i$ and ${\mathbf{\Lambda}^l_i}'$ be the corresponding sets of associated pairs with
respect to $H$ and $H'$ respectively; see $(\ref{p})$.
Define a map $\tau_{li}:\mathbf{\Lambda}^l_i\xra{\mathbf{\Lambda}^l_i}'$ such that for $(p, q)\in\mathbf{\Lambda}^l_i$
\begin{equation*}\tau_{li}((p,q))=\begin{cases}
(\widetilde{p}\aa, \widetilde{q}\aa) , & \text{if $p=\widetilde{p}\aa'$, $q=\widetilde{q}\aa'$ and $\aa'\in H'$};\\
(p, q),& \text{otherwise},
\end{cases}
\end{equation*} where $\aa\in H$ with $t(\aa)=t(\aa')$.
The map $\tau_{li}:\mathbf{\Lambda}^l_i\xra {\mathbf{\Lambda}^l_i}'$ is a bijection.
In fact, the inverse map of $\tau_{li}$ can be defined symmetrically.

Denote by $\mathcal{P}^{\bullet}=(\mathcal{P}^{l}, \d^{l})_{l\in \Z}$ and ${\mathcal{P}^{\bullet}}'=((\mathcal{P}^{l})', (\d^{l})')_{l\in \Z}$ the projective Leavitt complexes of $E$ with $H$ and $H'$
sets of associated edges, respectively.

Recall that $B=L_k(E^{\rm op})$. The maps $\phi': B\xra {\mathcal{P}^{\bullet}}', (p^{\rm op})^*q^{\rm op} \mapsto e_{t(p)}\z_{(p, q)}$ and
$\phi'_{\b}: B\xra {\mathcal{P}^{\bullet}}', (p^{\rm op})^*q^{\rm op}  \mapsto \delta_{t(p), s(\b)}\b\z_{(p, q)}$
with $(p, q)$ an associated pair and $\beta\in E^{1}$ are graded right $B$-module morphisms.

For each $l\in\Z$, define a $k$-linear map $\theta^l: \mathcal{P}^{l}\xra(\mathcal{P}^{l})'$ such that
$\theta^l(e_i\zeta_{(p, q)})=\phi'((p^{\rm op})^*q^{\rm op})$ and $\theta^l(\aa\zeta_{(p, q)})=\phi'_{\aa}((p^{\rm op})^*q^{\rm op})$ for $i\in E^{0}$ and
$(p, q)\in \mathbf{\Lambda}^{l}_{i}$ and $\aa\in E^1$
with $s(\aa)=i$.
Let $\theta^{\bullet}=(\theta^l)_{l\in \Z}:\mathcal{P}^{\bullet}\xra{\mathcal{P}^{\bullet}}'$.
By definitions we have $\theta^{\bullet}\circ\phi=\phi'$ and $\theta^{\bullet}\circ\phi_{\b}=\phi_{\b}'$ for each edge $\b\in E^1$.
Then we have $\theta^{\bullet}\circ\Phi=\Phi'$ with $\Phi'=\begin{pmatrix}
\phi'&(\phi'_{\b})_{{\b\in E^1}}
\end{pmatrix}:B\oplus B^{(E^1)}\xra {\mathcal{P}^{\bullet}}'$ a morphism of graded right $B$-modules.

\begin{prop}  $(1)$ The above map $\theta^{\bullet}:\mathcal{P}^{\bullet} \xra{\mathcal{P}^{\bullet}}'$ is an isomorphism of complexes of $A$-modules.

$(2)$ The above map $\theta^{\bullet}:\mathcal{P}^{\bullet} \xra {\mathcal{P}^{\bullet}}'$ is an isomorphism of right dg $B$-modules.
\end{prop}

\begin{proof} $(1)$  We can directly check that $\theta^l$ is an $A$-module map for each $l\in\Z$.
The inverse map of $\theta^l$ can be defined symmetrically. We have that $\theta^l$ is an isomorphism of $A$-modules for each $l\in\Z$.
It remains to prove that $\theta^{\bullet}$ is a chain map of complexes.
For each $i\in E^{0}$, $l\in \Z$ and
$(p, q)\in \mathbf{\Lambda}^{l}_{i}$, since $((\d^l)'\circ\theta^i)(\aa\z_{(p, q)})=0=(\theta^{l+1}\circ\d^l)(\aa\z_{(p, q)})$,
it suffices to prove that $((\d^l)'\circ\theta^l)(e_i\z_{(p, q)})=(\theta^{l+1}\circ\d^l)(e_i\z_{(p, q)})$
for $\aa\in E^1$ with $s(\aa)=i$. By \cite[Lemma 4.7]{li2}, we have that
$$((\d^l)'\circ\theta^l)(e_i\z_{(p, q)})=(\d^l)'(\phi'((p^{\rm op})^*q^{\rm op}))=\sum_{\{\b\in E^1\;|\; t(\b)=i\}}\phi'_{\b}(\aa^{\rm op}(p^{\rm op})^*q^{\rm op}))$$
and 
\begin{equation*}
\begin{split}
(\theta^{l+1}\circ\d^l)(e_i\z_{(p, q)})
&=\begin{cases}
\phi'_{\aa_1}(\aa_1^{\rm op}(p^{\rm op})^*q^{\rm op})), & \text{ if $p=\aa_1\widehat{p}$};\\
\sum_{\{\b\in E^1\;|\; t(\b)=i\}}\phi'_{\b}(\aa^{\rm op}(p^{\rm op})^*q^{\rm op})), & \text{ if $l(p)=0$.}
\end{cases}\\
&=\sum_{\{\b\in E^1\;|\; t(\b)=i\}}\phi'_{\b}(\aa^{\rm op}(p^{\rm op})^*q^{\rm op})).
\end{split}
\end{equation*} Then we are done.

$(2)$ It remains to prove that $\theta^{\bullet}$ is a graded right $B$-module morphism.
By Lemma \ref{pproj}, there exists a graded right $B$-module morphism $\Omega: \mathcal{P}^{\bullet}\longrightarrow B\oplus B^{(E^1)}$
such that $\Phi\circ\Omega=\rm{Id}_{\mathcal{P}^{\bullet}}$. Since $\theta^{\bullet}\circ \Phi=\Phi'$, we have that
$\theta^{\bullet}=\theta^{\bullet}\circ (\Phi\circ\Omega)=\Phi'\circ\Omega$ is a composition of graded right $B$-module morphisms.
\end{proof}

\section{Acknowledgements} The author thanks Professor Xiao-Wu Chen for many helpful discussions and encouragements. This project was supported by the National Natural Science Foundation of China (No.s 11522113 and 11571329). The author thanks Dr. Ambily Ambattu Asokan for hospitality when the author was in Cochin University of Science and Technology.  The author thanks Professor Roozbeh Hazrat for many helpful discussions and encouragements. The author also would like to acknowledge the support of Australian Research Council grant DP160101481. The author thanks the anonymous referees for their very helpful
suggestions to improve this paper.

\vskip 10pt

{\footnotesize\noindent Huanhuan ${\rm Li}^{1}$ \\
1. Centre for Research in Mathematics, Western Sydney University, Australia.\\
E-mail: h.li@westernsydney.edu.au}


\begin{thebibliography}{9999}

\bibitem{aajz} A. Alahmadi, H. Alsulami, S.K. Jain and E. Zelmanov, Leavitt path algebras of finite Gelfand-Kirillov dimension, J. Algebra Appl. $\mathbf{11}$ (6) (2012), 6pp.

\bibitem{ap} G. Abrams and G. Aranda Pino, The Leavitt path algebra of a graph, J. Algebra
$\mathbf{293}$ (2) (2005), 319-334.

\bibitem{amp} P. Ara, M.A. Moreno and E. Pardo, Nonstable $\mathbf{K}$-theory for graph algebras,
Algebr. Represent. Theory $\mathbf{10}$ (2) (2007), 157-178.







\bibitem{bu} R.O. Buchweitz, Maximal Cohen-Macaulay modules and Tate-cohomology over Gorenstein rings, unpublished manuscrip, 1987,
available at: http://hdl. handle. net/1807/16682.

\bibitem{bk}  A.I. Bondal,and M.M. Kapranov, Enhanced triangulated categories, Mat. Sb. $\mathbf{181}$ (5) (1990), 669-683; English transl.
Math. USSR-Sb. $\mathbf{70}$ (1) (1990), 93-107.

\bibitem{bn} M. B\"{o}kstedt and A. Neeman, Homotopy limits in triangulated categories, Compos. Math. $\mathbf{86}$ (1993), 209-234.




\bibitem{cy} X.W. Chen and D. Yang, Homotopy categories, Leavitt path algebras, and Gorenstein projective modules,
Int. Math. Res. Not. $\mathbf{10}$ (2015), 2597-2633.


\bibitem{cuntzkrieger} J. Cuntz and W. Krieger, A class of $C^*$-algebras and topological Markov chains, Invent. Math. $\mathbf{63}$ (1981), 25-40.


\bibitem{drozd}Ju.A. Drozd, Tame and wild matrix problems. In Representation theory, II(Proc.Second Internat. Conf., Carleton Univ., Ottawa, Ont., 1979), Lecture Notes in Math. 832, Springer-Verlag, Berlin 1980, 242-258.


\bibitem{grothendieck} A. Grothendieck, The cohomology theory of abstract algebraic varieties, in Proc. Internat.
Congress Math. (Edinburgh, 1958), Cambridge Univ. Press, New York, 1960, 103-118.

\bibitem{happel1} D. Happel, On the derived category of a finite-dimensional algebra, Comment. Math. Helv.,
$\mathbf{62}$ (1987), 339-389.

\bibitem{happel2} D. Happel, Triangulated categories in the representation theory of finite-dimensional algebras,
vol. 119 of London Mathematical Society Lecture Note Series, Cambridge University Press,
Cambridge, 1988.

\bibitem{hp} R. Hazrat, and R. Preusser, Applications of normal forms for weighted Leavitt path algebras: simple rings and domains, Algebr. Represent. Theory DOI: 10.1007/s10468-017-9674-3.

\bibitem{ke}  B. Keller, Deriving DG categories, Ann. Sci. \'{E}c.  Norm. Sup\'{e}r. (4)
$\mathbf{27}$ (1) (1994), 63-102.

\bibitem{kelly} G.M. Kelly, Chain maps inducing zero homology maps. Proc. Cambridge Philos. Soc. $\mathbf{61}$ (1965), 847-854.


\bibitem{kr} H. Krause, The stable derived category of a Noetherian scheme,
Compos. Math. $\mathbf{141}$ (2005), 1128-1162.

\bibitem{kprr} A. Kumjian, D. Pask, I. Raeburn, J. Renault, Graphs, groupoids, and
    Cuntz--Krieger algebras, J. Funct. Anal. $\mathbf{144}$ (1997), 505--541.



\bibitem{li} H. Li, The injective Leavitt complex, Algebr. Represent Theor. $\mathbf{21}$ (2018), 833-858.

\bibitem{li2} H. Li, The projective Leavitt complex, Proceedings of the Edinburgh Mathematical Society $\mathbf{61}$ (2018), 1155-1177.

\bibitem{n1} A. Neeman, The Grothendieck duality theorem via Bousfield's
techniques and Brown representability, J. Amer. Math. Soc. $\mathbf{9}$ (1996), 205-236.

\bibitem{neeman} A. Neeman, Triangulated categories, Annals of Mathematics Studies, vol. 148 (Princeton University Press, Princeton, NJ, 2001).





\bibitem{o} D.O. Orlov, Triangulated categories of sigularities and D-branes in Landau-Ginzburg models, Proc. Steklov Inst. Math. $\mathbf{(246)}$ (3) (2004), 227-248.

\bibitem{raeburn} I. Raeburn, Graph algebras. CBMS Regional Conference Series in
    Mathematics, 103.  The American Mathematical Society, Providence, RI, 2005.

 \bibitem{rickard1}J. Rickard, Derived categories and stable equivalence, J. Pure Appl. Algebra $\mathbf{61}$(1989), 303-317.
 
\bibitem{rickard2} J. Rickard, Morita theory for derived categories, J. London Math. Soc. (2), $\mathbf{39}$  (1989), 436-456.
 
\bibitem{rickard3} J. Rickard, Derived equivalences as derived functors, J. London Math. Soc. (2), $\mathbf{43}$ (1991), 37-48.

\bibitem{roiter} A.V. Roiter, Matrix problems, In Proceedings of the International Congress of Mathe-
maticians (Helsinki, 1978), Vol. 1, Acad. Sci. Fennica, Helsinki 1980, 319-322.


\bibitem{smith} S.P. Smith, Category equivalences involving graded modules over path algebras of quivers, Adv. Math. $\mathbf{230}$ (2012), 1780-1810.


\bibitem{verdier} J.-L. Verdier, Cat$\acute{\rm e}$gories d$\acute{\rm e}$riv$\acute{\rm e}$es, in SGA $4\frac{1}{2}$, vol. 569 of Lecture Notes in Mathematics,
Springer, Berlin, 1977, 262-311.





\end{thebibliography}
\end{document}